\documentclass{amsart}
\usepackage{graphicx}
\usepackage[utf8x]{inputenc}
\usepackage{amsmath}
\usepackage{amssymb}
\usepackage[dvipsnames]{xcolor}
\usepackage{amsthm}
\usepackage{thmtools}
\usepackage{thm-restate}
\usepackage{amsfonts}
\usepackage{graphicx}
\usepackage{stackrel}

\usepackage[hyperindex,colorlinks,breaklinks]{hyperref}
\usepackage[hyphenbreaks]{breakurl}
\usepackage[mathscr]{eucal}
\usepackage[shortlabels]{enumitem}
\usepackage{tikz}
\usepackage{tikz-cd}

\usepackage[style=alphabetic]{biblatex}
\newtheorem{theorem}{Theorem}
\newtheorem{theoremA}{Theorem}

\numberwithin{theorem}{section}
\numberwithin{equation}{section}

\newtheorem{lemma}[theorem]{Lemma}

\newtheorem{proposition}[theorem]{Proposition}
\newtheorem{remark}[theorem]{Remark}

\newcommand{\Z}{\mathbb{Z}}
\newcommand{\pZ}{\widehat{\Z}}

\newcommand{\hD}{\widehat{\Delta}}
\newcommand{\hG}{\widehat{\Gamma}}
\newcommand{\hM}{\widehat{M}}
\newcommand{\hN}{\widehat{N}}

\newcommand{\at}[1]{|_{#1}}

\title{The Profinite Rigidity of Torsion-Free Lamplighter Groups}
\author{Nikolay Nikolov and Julian Wykowski}
\address{Mathematical Institute, University of Oxford, Andrew Wiles Building, Radcliffe Observatory Quarter, Woodstock Road, Oxford OX2 6GG, United Kingdom}
\email{nikolay.nikolov@maths.ox.ac.uk}
\address{Department of Pure Mathematics and Mathematical Statistics, Centre for Mathematical Sciences, Wilberforce Road, Cambridge CB3 0WB, United Kingdom}
\email{jw2006@cam.ac.uk}
\date{\today}

\hypersetup{
    breaklinks=true,
	linkcolor={ForestGreen},
	citecolor={ForestGreen},
	urlcolor={ForestGreen}
}

\begin{document}
\begin{abstract}
We prove that the torsion-free lamplighter group $\Gamma = \Z^n \wr \Z$ of any rank $n \in \mathbb{N}$ is profinitely rigid in the absolute sense: the finite quotients of $\Gamma$ determine its isomorphism type uniquely among all finitely generated residually finite groups. The proof combines the theory of profinite rigidity for modules over Noetherian domains with an analysis of the algebraic properties of the lower central series of groups with the same profinite completion as $\Gamma$.
\end{abstract}
\maketitle
\section{Introduction}

A central theme in group theory over the last half-century has been the study of \emph{profinite rigidity}: can the structure of an infinite group be detected in its finite images? Let $\Gamma$ be a finitely generated residually finite group and write $\mathcal{C}(\Gamma)$ to denote the set of isomorphism classes of finite epimorphic images of $\Gamma$. We say that $\Gamma$ is \emph{profinitely rigid in the absolute sense} if for any finitely generated residually finite group $\Delta$ satisfying $\mathcal{C}(\Delta) = \mathcal{C}(\Gamma)$, there exists an isomorphism $\Delta \cong \Gamma$. This is equivalent to the isomorphism type of $\Gamma$ being distinguished by its profinite completion $\hG = \varprojlim{\mathcal{C}(\Gamma)}$ among all finitely generated residually finite groups \cite{Dixon1982}.

Originating in the work of Grothendieck on arithmetic geometry, profinite rigidity has been studied in contexts ranging from geometric topology to representation theory and computational algebra \cite{Grothendieck1970, Bridson2004, Bridson2015, Wilton2019}. Nonetheless, the question which groups are profinitely rigid in the absolute sense remains a vast mystery for most groups, the salient obstruction being that even if one possesses a lot of information about a group $\Gamma$, one cannot in general exclude the possibility that some unknown exotic group $\Delta$ could have the same finite images as $\Gamma$. In fact, for free and free solvable groups, the question of profinite rigidity is a long-standing open problem attributed to Remeslennikov \cite[Question 15]{Noskov1982}.

Recently, the second author developed a theory of profinite rigidity for modules over Noetherian domains \cite{Wykowski2025_2}, which was then harnessed to establish profinite rigidity for two infinite classes of metabelian groups, namely, solvable Baumslag--Solitar groups \cite{Wykowski2025_2} and free metabelian groups \cite{Wykowski2025_3}. The purpose of this article is to expand this theory further to establish the absolute profinite rigidity of \emph{torsion-free lamplighter groups}, i.e. the groups
\begin{equation} \label{Eq::DefOfLL}
\Gamma_n = \Z^n \wr \Z \cong M \rtimes \langle t\rangle
\end{equation}
where $n \in \mathbb{N}$ and $M$ is the free $\Z[\langle t\rangle]$-module of rank $n$. We answer the question of absolute profinite rigidity positively for this class of groups. Torsion-free lamplighter groups occupy a central position in geometric group theory and dynamics: see \cite{Genevois2024} for a survey. The group $\Gamma_1 = \Z \wr \Z$ is particularly important in the study of groups acting homeomorphically on 1-manifolds \cite{Cannon1996,Balasubramanya2024}.
\begin{restatable}{theoremA}{MainThm}
    \label{Thm::MainThm}
    Let $n \in \mathbb{N}$ be a positive integer. The torsion-free lamplighter group $\Gamma_n = \Z^n \wr \Z$ is profinitely rigid in the absolute sense.
\end{restatable}
We note that Theorem~\ref{Thm::MainThm} contrasts with \cite{Pickel1974}, where metabelian groups with infinite profinite genera were constructed, and \cite{Nikolov2021}, where four-step solvable groups with uncountable profinite genera were found. On the other hand, for finitely generated lamplighter groups with base of prime exponent, i.e. groups of the form $\mathbb{F}_p^k \wr \Z$ for some $k \in \mathbb{N}$ and prime number $p$, the analogous question was answered positively in \cite{Blachar2025}.

The proof is outlined as follows. Using the decomposition (\ref{Eq::DefOfLL}), one aims to reduce the problem to the question of profinite rigidity of $M$ within the category of $\Z[\langle t \rangle]$-modules, which one can then solve using the theory of profinite rigidity for modules over Noetherian domains \cite{Wykowski2025_2}. However, the difficulty lies in the fact that $M$ is strictly larger than the derived subgroup of $\Gamma$, so one must first find a sensible counterpart $N$ to the module $M$ among the subgroups of an arbitrary finitely generated residually finite group $\Delta$ with $\hD \cong \hG$. Having explored some preliminary results in Section~\ref{Sec::Prelims}, the subgroup $N \leq \Delta$ is constructed in Section~\ref{Sec::Nilpotent} relative to the infinite family of maximal nilpotent quotients of $\Delta$. An essential ingredient at this step is a theorem of the first author on the closure of verbal subgroups in profinite group \cite[Theorem 1.4]{Nikolov2007}. In Section~\ref{Sec::Lifting}, the defining relative properties of $N$ are then upgraded to global characteristic properties. Finally, in Section~\ref{Sec::Final}, the proof of Theorem~\ref{Thm::MainThm} is completed.

\section*{Acknowledgements}

The authors are grateful to Gareth Wilkes for fruitful discussions and insight. The authors would also like to thank the Isaac Newton Institute for Mathematical Sciences, Cambridge, for support and hospitality during the programme Discrete and Profinite Groups, where this collaboration was initiated. This programme was supported by EPSRC grant EP/Z000580/1. In addition, the second author is grateful for the financial support of a Cambridge International Trust \& King's College PhD Scholarship at the University of Cambridge.

\section{Preliminary Results}\label{Sec::Prelims}
In this section, we survey and prove certain preliminary results that shall be utilised in the proof of Theorem~\ref{Thm::MainThm}. This includes results from the theory of profinite rigidity over Noetherian domains, nilpotent quotients of finitely generated profinite groups, and a brief lemma regarding orders of elements in finite product groups.
\subsection{Profinite Rigidity over Noetherian Domains}
The first main ingredient in the proof of Theorem~\ref{Thm::MainThm} will be the theory of profinite rigidity for modules over Noetherian domains developed by the second author in \cite{Wykowski2025_2}. Let $\Lambda$ be a finitely generated commutative domain, that is, a quotient of the polynomial ring $\Z[x_1, \ldots, x_n]$ for some $n \in \mathbb{N}$ by a prime ideal. Such a domain $\Lambda$ is automatically Noetherian by the Hilbert Basis Theorem, and all finitely generated $\Lambda$-modules are automatically residually finite \cite[Theorem 1]{Orzech1970}. We say that $\Lambda$ is \emph{homologically taut} if all finitely generated projective $\Lambda$-modules are free.

Given a $\Lambda$-module $M$, we define the \emph{profinite completion} $\hM$ of $M$ as the profinite $\Lambda$-module given by the inverse limit
\[
\hM = \varprojlim_{N \trianglelefteq_f M} (M/N)
\]
in the category of $\Lambda$-modules, where $N\trianglelefteq_f M$ ranges through the finite-index submodules of $M$. A finitely generated $\Lambda$-module $M$ is then said to be \emph{$\Lambda$-profinitely rigid} if for any finitely generated $\Lambda$-module $N$, the equivalence $\hM \cong \hN \Longleftrightarrow M \cong N$ holds. We note the following profinite rigidity result for free modules over homologically taut Noetherian domains.
\begin{theorem}[Theorem C in \cite{Wykowski2025_2}]\label{Thm::NoetherianDomains}
    Let $\Lambda$ be a finitely generated Noetherian domain. If $\Lambda$ is homologically taut then all finitely generated free $\Lambda$-modules are $\Lambda$-profinitely rigid. 
\end{theorem}

We shall use this theorem in conjunction with the following celebrated result of Quillen and Suslin, which we cite here in localised form.
\begin{theorem}[Corollary 7.4 in \cite{Suslin1977}]\label{Thm::Quillen-Suslin}
    The ring of integral Laurent polynomials $\Lambda = \Z[x^{\pm}_1, \ldots, x^{\pm}_n]$ in $n$ variables is homologically taut for any $n \in \mathbb{N}$.
\end{theorem}

\subsection{The Profinite Invariance of Nilpotent Quotients}
Given a discrete or profinite group $G$, we shall write $\gamma_n(G)$ for the $n$-th term in its lower central series, that is
\[
\gamma_n(G) = [G,\gamma_{n-1}(G)]
\]
when $n \geq 2$ and $\gamma_1(G) = G$. In the case where $G$ is a profinite group, the commutator subgroup denotes the abstract subgroup generated by all commutators of the given form, and so $\gamma_n(G)$ is not a priori a closed subgroup of $G$. However, for finitely generated profinite groups, it does indeed turn out to be a closed subgroup, as follows from a deep theorem due to the first author and Dan Segal.
\begin{theorem}[Theorem 1.4 in \cite{Nikolov2007}]\label{Thm::NS}
    Let $G$ be a finitely generated profinite group. For any $k \in \mathbb{N}$, the abstract subgroup $\gamma_k(G)$ is closed in $G$.
\end{theorem}

A consequence of Theorem~\ref{Thm::NS} is the following proposition, which shall prove essential in the construction of the characteristic subgroup in  Section~\ref{Sec::Nilpotent}.

\begin{proposition}\label{Prop::ProfiniteInvarianceOfCentralSeries}
    Let $\Gamma, \Delta$ be finitely generated residually finite groups. If there is an isomorphism of profinite completions $\hG \cong \hD$ then there is an isomorphism of lower central factors $\gamma_n(\Gamma)/\gamma_{n+1}(\Gamma) \cong \gamma_n(\Delta) / \gamma_{n+1}(\Delta)$ for all $n \in \mathbb{N}$. 
\end{proposition}
\begin{proof}
   We claim first that for any finitely generated residually finite group $\Gamma$, the closure $\overline{\gamma_n(\Gamma)}$ in the profinite completion $\hG$ satisfies \begin{equation}\label{Eq::gamma_closure}
       \overline{\gamma_n(\Gamma)} = \gamma_n(\hG)
   \end{equation} for all $n \in \mathbb{N}$. Indeed, for any discrete or profinite group $G$, the $n$-th lower central term $\gamma_n(G)$ is equal to the verbal subgroup in $G$ associated to the word $w_n \colon G^n \to G $ given by the simple commutator $w_n(x_1, \ldots, x_n) = [x_1, \ldots, x_n]$. The inclusion $\overline{\gamma_n(\Gamma)} \subseteq \gamma_n(\hG)$ holds as $\gamma_n(\hG)$ is closed by Theorem~\ref{Thm::NS}. Conversely, as $\Gamma$ is dense in $\hG$, the image $w_n(\Gamma^n)$ under the continuous map $w_n$ is dense in the closed set $\gamma_n(\hG) = w_n(\hG^n)$. Hence the subgroup $\gamma_n(\Gamma) = \langle w_n(\Gamma^n)\rangle$ is dense in the closed subgroup $\gamma_n(\hG) = \langle w_n(\hG) \rangle$, whence $\overline{\gamma_n(\Gamma)} = \gamma_n(\hG)$ holds. This proves the claim.

   Now, write $\overline{\gamma_n(\Gamma)},\overline{\gamma_{n+1}(\Gamma)}$ for the closures of $\gamma_n(\Gamma),\gamma_{n+1}(\Gamma)$ within the profinite completion $\hG$, respectively, and write $\overline{\gamma_n(\Gamma)/\gamma_{n+1}(\Gamma)}$ for the closure of the quotient $\gamma_n(\Gamma)/\gamma_{n+1}(\Gamma)$ within the profinite completion $\widehat{\Gamma_{n+1}}$ of the nilpotent group $\Gamma_{n+1} = \Gamma / \gamma_{n+1}(\Gamma)$. Using the third isomorphism theorem and the fact that the inverse limit functor is exact \cite[Proposition 2.2.4]{Ribes2010}, we find that
\begin{align}\label{Eq::SecondIsomorphismThm}
\overline{\gamma_n(\Gamma)}/\overline{\gamma_{n+1}(\Gamma)} &\cong \varprojlim_{ U\trianglelefteq_f \Gamma} \left(\frac{\gamma_n(\Gamma)/U \cap \gamma_n(\Gamma)}{\gamma_{n+1}(\Gamma)/ U \cap  \gamma_{n+1}(\Gamma)}\right) \\
&\cong \varprojlim_{U\trianglelefteq_f \Gamma} \left(\frac{\gamma_n(\Gamma)\cdot U}{\gamma_{n+1}(\Gamma) \cdot  U}\right) \\
&\cong \overline{\gamma_n(\Gamma)/\gamma_{n+1}(\Gamma)}
   \end{align}
holds. Moreover, finitely generated nilpotent groups induce the full profinite topology on all their subgroups \cite[p. 223]{Segal1983}, so 
\begin{equation}
\overline{\gamma_n(\Gamma)/\gamma_{n+1}(\Gamma)} \cong \widehat{\gamma_n(\Gamma)/\gamma_{n+1}(\Gamma)}
\end{equation}
is in fact the profinite completion of the abelian group $\gamma_n(\Gamma)/\gamma_{n+1}(\Gamma)$. Combining this with the analogous equation for $\Delta$ and the claim, we obtain
\begin{equation}
\widehat{\gamma_n(\Gamma)/\gamma_{n+1}(\Gamma)} \cong \gamma_n(\hG) / \gamma_{n+1}(\hG) \cong \gamma_n(\hD) / \gamma_{n+1}(\hD) \cong \widehat{\gamma_n(\Delta)/\gamma_{n+1}(\Delta)}
\end{equation}
for all $n \in \mathbb{N}$. But finitely generated abelian groups are profinitely rigid \cite[Corollary 3.2.12]{Wilkes2024}, so we must in fact have $\gamma_n(\Delta)/\gamma_{n+1}(\Delta) \cong \gamma_n(\Gamma)/\gamma_{n+1}(\Gamma)$, as postulated.
\end{proof}

\subsection{Orders in Finite Products} Finally, we state and prove the following elementary lemma regarding orders of elements in finite products groups, which we shall appeal to in Section~\ref{Sec::Lifting}.

\begin{lemma} \label{nocyclic} Let $A,B$ be two finite abelian groups and let $x=(a,b)$ and $y=(1,b')$ be two elements in $A \times B$. Assume that there is a prime number $p$ and positive integers $n,m,k$ with $n \geq m$ satisfying
\[
\operatorname{ord}(a)=p^n, \ \operatorname{ord}(b)=p^m, \ \operatorname{ord}(b')=p^k
\]
where $\operatorname{ord}(-)$ denotes the order of an element. Then the group $\langle x,y \rangle$ generated by $x$ and $y$ in $A \times B$ admits an isomorphism $\langle x,y \rangle \cong C_{p^n} \times C_{p^k}$.
\end{lemma}
\begin{proof} Let $D= \langle x,y\rangle$ and note that $\operatorname{ord}(x)=p^n$ and $ \operatorname{ord}(y)=p^k$ holds. Therefore, it will be sufficient to show that $|D| \geq p^{n+k}=\operatorname{ord}(x) \cdot \operatorname{ord}(y)$, as then we obtain postulated isomorphism $D \cong \langle x \rangle  \times \langle y \rangle \cong C_{p^n} \times C_{p^k}$.
Denote by $\pi_A: A \times B \rightarrow A$ the projection onto $A$. We have $|\pi(D)|=|\langle a \rangle|=p^n$ while $\ker \pi \cap D= B \cap D$ contains $b'$. Hence $|D| = |\pi(D)| | D \cap B| \geq \operatorname{ord}(a) \cdot \operatorname{ord}(b') \geq p^n p^k$ and the result follows.
\end{proof}

\section{Nilpotent Quotients}\label{Sec::Nilpotent}

In this section, we commence with the proof of Theorem~\ref{Thm::MainThm} on the profinite rigidity of the torsion-free lamplighter groups $\Gamma_n = \Z^n \wr \Z$. The main result of this section shall be the following proposition, wherein we outline an algorithm for constructing a subgroup $N \leq \Delta$ for any group $\Delta$ in the profinite genus of $\Gamma$ which corresponds to the module $M \leq \Gamma_n$. We delay the proof that $N$ is abelian and forms a $\Z[\langle t \rangle]$-module until the following section.
\begin{proposition} \label{nilq}
    Let $n \in \mathbb{N}$ be a positive integer, let $\Gamma = \Z^n \wr \Z$ be the associated torsion-free lamplighter group, and let $\Delta$ be a finitely generated residually finite group satisfying $\hD \cong \hG$. There exists a subgroup $N \trianglelefteq \Delta$ such that $\Delta / N \cong \Z$ and for any positive integer $k \in \mathbb{N}$, the inclusion $[N,[\gamma_k(\Delta),\gamma_k(\Delta)]] \leq \gamma_k(\Delta)$ holds.
\end{proposition}

To prove Proposition~\ref{nilq}, we shall need several preliminary results. Indeed, let $n \in \mathbb{N}$ be a positive integer and $\Gamma = \Z^n \wr \Z$ be the associated torsion-free lamplighter group. Suppose that $\Delta$ is a finitely generated residually finite group admitting an isomorphism of profinite completions $\widehat{\Delta} \cong \widehat{\Gamma}$. Using Proposition~\ref{Prop::ProfiniteInvarianceOfCentralSeries} and the lower central series of $\Gamma$, we obtain the following.

\begin{lemma}\label{Lem::CentralSeries}
    There are isomorphisms $\Delta^\mathrm{ab} \cong \Z^{n+1}$ and $\gamma_m(\Delta)/\gamma_{m+1}(\Delta) \cong \mathbb Z^{m}$ for all positive integers $m > 1$. 
\end{lemma} 

Some further restriction on $\Delta$ is given by the following lemma.

\begin{lemma} \label{cyclic} Let $U \trianglelefteq \Delta$ be a normal subgroup satisfying $U \leq [\Delta, \Delta]$ and whose index in $[\Delta, \Delta ]$ is finite. Consider the finite $\mathbb Z[\Delta^\mathrm{ab}]$-module $T =[\Delta,\Delta]/U$. The image of the natural map $\Delta \twoheadrightarrow \Delta / C_\Delta(T) \to \operatorname{Aut}_\mathbb Z(T)$ is a finite cyclic group.
\end{lemma}
\begin{proof}
    To begin, note that the analogous result for $\Gamma$ follows from the decomposition (\ref{Eq::DefOfLL}) and the observation that $M = C_\Gamma([\Gamma,\Gamma])$. Now let $U \trianglelefteq \Delta$ be a normal subgroup contained in $D = [\Delta, \Delta]$ with finite index and let $E = [\Gamma,\Gamma]$. Write $\Lambda = \Z[\Delta^\mathrm{ab}] \cong \Z[\Gamma^\mathrm{ab}]$ and consider $D$ and $E$ as $\Lambda$-modules where $\Lambda$ acts via conjugacy within $\Delta$ and $\Gamma$, respectively. By \cite[Lemma 5.2]{Wykowski2025_2}, the closures $\overline{D}$ and $\overline{E}$ within the respective profinite completions $\hD$ and $\hG$ are isomorphic to the completions $\widehat{D}$ and $\widehat{E}$ as $\Lambda$-modules. By assumption, there exists an isomorphism of profinite groups $\phi \colon \hD \to \hG$, so we obtain a commutative diagram
    \[
    \begin{tikzcd}
0 \arrow[r] & {\widehat{D} \cong \overline{[\Delta, \Delta]}} \arrow[r] \arrow[d, "f"] & \hD \arrow[r] \arrow[d, "\phi"] & \hD^\mathrm{ab} \cong \widehat{\Z}^{n+1} \arrow[r] \arrow[d, "\phi^\mathrm{ab}"] & 0 \\
0 \arrow[r] & {\widehat{E} \cong \overline{[\Gamma, \Gamma]}} \arrow[r]                & \hG \arrow[r]                   & \hG^\mathrm{ab} \cong \widehat{\Z}^{n+1} \arrow[r]                               & 0
\end{tikzcd}
    \]
with exact rows and whose columns are isomorphisms of profinite abelian groups. Note that the abelianised map $\phi^\mathrm{ab}$ induces an automorphism $\alpha \colon \widehat{\Lambda} \xrightarrow{\sim}\widehat{\Lambda}$ as a profinite ring. Moreover, the isomorphism of infinitely generated profinite abelian groups $f \colon \widehat{D} \xrightarrow{\sim} \widehat{E}$ is not necessarily a morphism of $\Lambda$-modules, but satisfies the twisted identity
\begin{equation}
    f(\lambda \cdot m) = \alpha(\lambda) \cdot f(m) 
\end{equation}
whenever $\lambda \in \widehat{\Lambda}$ and $m \in \widehat{D}$. As $U$ is normal in $\Delta$, it forms a submodule for the conjugacy $\Lambda$-module structure, and $T = D/U$ forms a finite $\Lambda$-module quotient of the $\Lambda$-profinite completion $\widehat{D}$. The twisted isomorphism $f \colon \widehat{D} \to \widehat{E}$ hence descends to an isomorphism of finite abelian groups
\[
\widetilde{f} \colon \widehat{D}/\overline{U} \xrightarrow{\sim} \widehat{E}/f(\overline{U})
\]
which again satisfies the twisted identity
\begin{equation}\label{Eq::Twist}
    \widetilde f(\lambda \cdot m) = \alpha(\lambda) \cdot\widetilde f(m) 
\end{equation}
whenever $\lambda \in \widehat{\Lambda}$ and $m \in \widehat{D}/\overline{U}$. As observed before, the analogue of the postulated result for $\Gamma$ holds true, so the image 
\[
I_\Gamma := \operatorname{Im}\left(\hG^\mathrm{ab} \to \operatorname{Aut}_\Z(\widehat E/f(\overline U))\right) = \operatorname{Im}\left(\Gamma^\mathrm{ab} \to \operatorname{Aut}_\Z(\widehat E/f(\overline U)))\right)
\]
is a finite cyclic group. But (\ref{Eq::Twist}) gives
\[
\operatorname{Im}\left(\Delta^\mathrm{ab} \to \operatorname{Aut}_\Z(T)\right) = \operatorname{Im}\left(\widehat \Delta^\mathrm{ab} \to \operatorname{Aut}_\Z(\widehat D/\overline U)\right) = \alpha(I_\Gamma) = \phi^\mathrm{ab}(I_\Gamma)
\]
which is then also a finite cyclic group, as postulated.
\end{proof}

We now commence with the proof of Proposition~\ref{Prop::LiftingCentralisers}. Given a positive integer $m > 2$, let $\Delta_m=\Delta/\gamma_m(\Delta)$ and consider the group $V_m=[\Delta_m,\Delta_m]=\gamma_2(\Delta)/\gamma_m(\Delta)$ which is abelian as $\hD \cong \hG$ and $
\Gamma$ is metabelian. Moreover, Lemma~\ref{Lem::CentralSeries} yields $V_m \cong \Z^{n(m-2)}$ for any positive integer $m >2$. The group $\Delta_m^\mathrm{ab}$ acts by conjugation on $V_m$ fixing the lower central factors $\gamma_s(\Delta_m)/\gamma_{s+1}(\Delta_m) \cong \Z^n$ set-wise for all $s=2,3, \ldots ,m-1$ and hence its image $L_m$ in $\operatorname{Aut}(V_m)$ is a subgroup of the group $\operatorname{UT}_{n(m-2)}(\Z)$ of unitrianlgular matrices with respect to an appropriate choice of basis of $V_m$. Now $\operatorname{UT}_{n(m-2)}(\Z)$ is a finitely generated nilpotent group, so it induces the full profinite topology on $L_m$ by \cite[pp. 223]{Segal1983}. In turn, this topology coincides with its congruence topology (defined by the principal congruence subgroups $\operatorname{UT}_{n(m-2)}(k \Z)$ for positive integers $k$) by \cite[Theorem 5]{Formanek1976}. We conclude that any finite image of $L_m$ is a congruence image, meaning that it factors through a principal congruence image $L_m/L_m(k)$ for some $k \in \mathbb{N}$, where $L_m(k)=L \cap \operatorname{UT}_{n(m-2)}(k \Z)$ consists of all matrices in $L$ that are congruent to $\mathrm{Id}_{n(m-2)}$ mod $k\Z$.

Now let $T_m = C_\Delta(V_m)$ be the preimage of the centraliser of $V_m$ under the map $\Delta \twoheadrightarrow \Delta_m$. We claim that $T_m = T_4$ for all $m \geq 4$ and $\Delta / T_4 \cong \Z$ holds. Indeed, choose any positive integer $m \geq 4$. By the preceding argument, all finite quotients of the group $\Delta / T_m \cong L_m$ are of the form
\[
\frac{L_m}{L_m(k)} \cong \frac{\Delta_m}{C_{\Delta_m}(V_m/kV_m)}
\] which is a cyclic group by Lemma~\ref{cyclic}. Via \cite[Corollary 3.2.12]{Wilkes2024}, it follows that $L_m$ itself is cyclic. Moreover, $T_m \leq T_4$, so $L_m \cong \Delta / T_{m} \twoheadrightarrow \Delta / T_4 \cong L_4$. But the only cyclic group surjecting onto $\Z$ is $\Z$ itself, and $\Z$ is Hopfian. Thus, the first statement in the claim follows from the second, and to prove the claim, it will suffice to show that $\cong \Delta / T_4$ is infinite. 

To see this, choose any odd prime number $p$. By assumption, $\Delta$ has the same finite images as $\Gamma$, so in particular we obtain an epimorphism 
\[
\Delta_4 \twoheadrightarrow Q_p = \left(\frac{\mathbb{F}_p[x]}{(x-1)^3}\right)^n \rtimes C_p
\] where the generator of $C_p$ acts via multiplication by $x$. The action of $x$ on the commutator $[Q_p, Q_p ] = \left(\frac{(x-1)}{(x-1)^3}\right)^n$ is non-trivial, so in particular $L_4 \twoheadrightarrow C_p$. As this holds for any odd prime $p$, the group $L_4 \cong \Delta / T_4 $ must be infinite, and the claim follows.

To complete the proof of the proposition, define $N = T_4$. By the first part of the claim, we have $N = T_m = C_\Delta(V_m)$ for all $m \geq 4$, so $N$ centralizes $[\Delta, \Delta]$ modulo $\gamma_m(\Delta)$ for each $m \in \mathbb{N}$. On the other hand, the second part of the claim yields $\Delta / N \cong \Z$. This completes the proof of  Proposition~\ref{nilq}.

\begin{remark}\label{Rem::ModuleDescription}
    Let $N \trianglelefteq \Delta$ be as in Proposition~\ref{nilq} and choose $\zeta \in \Delta$ such that $\langle \zeta, N \rangle = \Delta$. For any $m > 2$, there exists an isomorphism of $\Z[x^\pm]$-modules
    \[
    \frac{\gamma_2(\Delta)}{\gamma_m(\Delta)} \cong \left(\frac{\Z[x]}{(x-1)^{m-2}}\right)^n
    \]
    where $x$ acts as conjugation by $\zeta$ on the left and via multiplication on the right.
\end{remark}

\begin{proof}
    By Proposition~\ref{nilq}, there exists $\zeta \in \Delta$ whose image generates $\Delta/N$. As demonstrated in the proof of Proposition~\ref{nilq}, for any $m >2$, there is an isomorphism of graded abelian groups
    \[
    \frac{\gamma_2(\Delta)}{\gamma_m(\Delta)} \cong \bigoplus_{i=2}^{m-1} \frac{\gamma_{i}(\Delta)}{\gamma_{i+1}(\Delta)} \cong \bigoplus_{i=2}^{m-1} \Z^n
    \]
    and the conjugation action of $\Delta$ is upper triangular with respect to the grading. Choose elements $\{\alpha_{1,1}, \ldots, \alpha_{n,1}\} \subset \gamma_2(\Delta)$ which are a basis for the free abelian group $\frac{\gamma_2(\Delta)}{\gamma_3(\Delta)}$. Since the ranks of the free abelian groups $\frac{\gamma_2(\Delta)}{\gamma_3(\Delta)}$ and $\frac{\gamma_3(\Delta)}{\gamma_4(\Delta)}$ are equal we conclude that the elements $\alpha_{i,2}:= [\zeta, \alpha_{i,1}]= \zeta \alpha_{i,1} \zeta^{-1} -a_{i,1}$ for $i=1, \ldots, n$, are a basis for $\frac{\gamma_3(\Delta)}{\gamma_4(\Delta)}$. Continuing in the same way we define  $\alpha_{i,j+1}:=[\alpha_{i,j}, \zeta]$  for all $1 \leq i \leq n$ and $1 \leq j< m-2$, and note that $\zeta \alpha_{i,m-2} \zeta^{-1} = \alpha_{i,m-2}$ holds for all $1 \leq i \leq n$. The free abelian group $\frac{\gamma_2(\Delta)}{\gamma_m(\Delta)}$ has basis $\{\alpha_{i,j} \ | \ 1 \leq i \leq n, 1 \leq j \leq m-2 \}$ and acquires the structure of a $\Z[x]$-module whereby $x$ acts as conjugation by $\zeta$ within $\Delta$. Consider now the homomorphism of $\Z[x]$-modules \[
    \phi \colon \Z[x]^n \to \frac{\gamma_2(\Delta)}{\gamma_m(\Delta)}
    \]
    generated by mapping the $i$-th basis vector $e_i$ of $\Z[x]^n$ to $\phi(e_i) = \alpha_{i,1}$. This mapping is surjective as $\alpha_{i,j} = \phi\left((x-1)^{j-1} \alpha_{i,1}\right)$ holds for all $i$ and $j > 1$. On the other hand, one finds that $\operatorname{Ker}(\phi) = (x-1)^{m-2}$. Thus $\phi$ descends to the postulated isomorphism. 
\end{proof}

\section{Lifting Centralisers}\label{Sec::Lifting}

In this section, we upgrade Proposition~\ref{nilq} from a relative statement in maximal nilpotent quotients to a global statement about any group $\Delta$ in the profinite genus of $\Gamma$. Concretely, we prove the following result. 
\begin{proposition}\label{Prop::LiftingCentralisers}
    Let $n \in \mathbb{N}$ be a positive integer, let $\Gamma = \Z^n \wr \Z$ be the associated torsion-free lamplighter group, and let $\Delta$ be a finitely generated residually finite group satisfying $\hD \cong \hG$. There exists an abelian normal subgroup $N \trianglelefteq \Delta$ which centralises the commutator $[\Delta,\Delta]$ in $\Delta$ and satisfies $\Delta / N \cong \Z$.
\end{proposition}
We split the proof of Proposition~\ref{Prop::LiftingCentralisers} into the two components of its statement: see Lemma~\ref{centralizer} and Lemma~\ref{Lem::Abelian} below. Note that Lemma~\ref{centralizer} would follow immediately from Proposition~\ref{nilq} if one knew that $\Delta$ is residually nilpotent; however, unlike for $\Gamma$, there is no reason a priori to believe that $\Delta$ satisfies this condition.

Let $n \in \mathbb{N}$ be a positive integer, let $\Gamma = \Z^n \wr \Z$ be the associated torsion-free lamplighter group, and let $\Delta$ be a finitely generated residually finite group satisfying $\hD \cong \hG$. Let $N \trianglelefteq \Delta$ be the subgroup provided by Proposition~\ref{nilq}.

\begin{lemma} \label{centralizer} The subgroup $N \trianglelefteq \Delta$ centralises the derived subgroup $\gamma_2(\Delta)$ of $\Delta$.
\end{lemma}
\begin{proof}
Let $D= \gamma_2(\Delta) = [\Delta, \Delta]$ which we consider as a module over the Noetherian domain $ \Z[\Delta^\mathrm{ab}]$ whereby $\Delta^\mathrm{ab}$ acts via conjugation in $\Delta$. By Proposition~\ref{nilq}, the quotient $\Delta / N$ is cyclic, so we may choose an element $\zeta \in \Delta$ such that $\langle N,\zeta\rangle = \Delta$.

As $\Delta / D \cong \Delta^\mathrm{ab} \cong \Gamma^\mathrm{ab} \cong \Z^n$ is nilpotent, a classic result of P. Hall implies \cite[Exercise 15.4.2]{Robinson1996} that $\bigcap_{p \textrm{ prime}}pD = 0$. Moreover, the $\Z[\Delta^\mathrm{ab}]$-module $D/pD$ is residually finite for each prime $p$ as $\Z[\Delta^\mathrm{ab}]$ is a finitely generated commutative ring \cite[Theorem 1]{Orzech1970}. Hence it will suffice to show that $N$ centralizes any finite $\mathbb F_p[\Delta^\mathrm{ab}]$-module of the form $V= D/L$ with $L \geq pD$. Indeed, choose any prime $p$ and let $V = D / L$ be a finite $\mathbb{F}_p[\Delta^\mathrm{ab}]$-module with $L \geq pD$. We shall write $\rho_V \colon \Delta \to \Delta^\mathrm{ab} \to \operatorname{GL}(V)$ for the conjugation action map.

\textbf{Claim:} For any $\nu \in N$, the order of $\rho_V(\nu) \in \operatorname{GL}(V)$ is a power of $p$. Indeed, suppose for a contradiction that there exists $\nu \in N$ such that the order of $\rho_V(\nu) \in \operatorname{GL}(V)$ is divisible by some prime $q \neq p$ with positive multiplicity. We may write $\operatorname{ord}(\rho_V(\zeta)) = q^lr$ for some $l \geq 0$ and $r \in \mathbb{N}$ coprime to $q$. The action of $\Delta^\mathrm{ab}$ fixes $\gamma_m(\Delta)$ for each $m \in \mathbb{N}$, so the subgroup $\gamma_m(\Delta) \trianglelefteq D$ forms a $\Z[\Delta^\mathrm{ab}]$-submodule. Consider hence the $\mathbb{Z}[\Delta^\mathrm{ab}]$-module \[
W=\frac{D}{qD+\gamma_{q^l+2}(\Delta)} \cong \frac{\mathbb{F}_q[x]}{(x - 1)^{q^l}}
\]
where the isomorphism derives from Remark~\ref{Rem::ModuleDescription}. Write $\rho_W \colon \Delta \to \Delta^\mathrm{ab} \to \operatorname{GL}(W)$ for the conjugation action map. By Proposition~\ref{nilq}, the element $\nu \in N$ acts trivially on $W$, so $\operatorname{ord}(\rho_W(\nu)) = 1$. On the other hand, $\zeta$ acts as multiplication by $x$, so $ \operatorname{ord}(\rho_W(\zeta)) = q^l$. Now $p$ and $q$ are coprime, so we obtain a direct sum decomposition
\[
T := \frac{D}{\left(qD+\gamma_{q^l+2}(\Delta)\right) \cap L} \cong V \oplus W
\]
as $\Z[x^\pm]$-modules. Writing $ A = \operatorname{GL}(W)$ and $B = \operatorname{GL}(V)$ as well as $a = \rho_W(\zeta^r)$ and $b = \rho_V(\zeta^r)$ and $b' = \rho_V(\nu)^{\operatorname{ord}(\rho_V(\nu))/q}$, we find that $\operatorname{ord}(a) = \operatorname{ord}(b) = q^l$ while $\operatorname{ord}(b') = q$. It follows via Lemma~\ref{nocyclic} that the images of $\zeta^r$ and $\nu^{\rho_V(\nu))/q}$ in $\operatorname{Aut}(T) \cong A \times B$ generate a subgroup of $\operatorname{Aut}(T)$ isomorphic to $C_{q^l} \times C_q$, a contradiction to Lemma~\ref{cyclic}. This proves the claim.

To complete the proof of Lemma~\ref{centralizer}, we want to show that $N$ centralises $V$. Indeed, suppose for a contradiction that there exists $\nu \in N$ acting non-trivially on $V$. By the claim, we have $\operatorname{ord}(\rho_V(\nu)) = p^k$ for some positive integer $k \geq 1$. Let $m(x) \in \mathbb F_p[x]$ be the minimal polynomial of $\zeta$ acting on $V$. We may factorise it as $m(x)=m_0(x) m_1(x)$ where $m_0(x)=(x-1)^d$ with $ d \geq 0$ and $m_1(1) \neq 0$, that is, $m_1(x)$ is coprime to $(x-1)$. Using the primary decomposition theorem, we obtain a decomposition of $\zeta$-invariant vector spaces \[V=V_0 \oplus V_1\] such that $m_i$ is the minimal polynomial of $\zeta$ on $V_i$ for $i = 0,1$. Since $N$ and $\zeta$ commute in their action on $V$ we deduce that each $V_i$ is in fact a submodule for $\Z[\Delta^\mathrm{ab}]$. By construction, there exists a power of $(x-1)$ that annihilates $V_0$. By the claim, $\rho_V(N)$ is a finite $p$-group, so its action on $V$ must also be unipotent. Since again the action of $\rho_V(N)$ and $\rho_V(\zeta)$ on $V$ commute, we can infer that the action of $\Delta^\mathrm{ab}$ on $V$ is also unipotent, i.e.
\[
\frak{I}^s V_0 = 0 
\] where $\frak{I}$ is the augmentation ideal of $\mathbb{F}_p[\Delta^\mathrm{ab}]$ and $s \in \mathbb{N}$ is a positive integer. In particular, $V_0$ is a homomorphic image of $D/\frak{I}^s D = \gamma_2(\Delta)/\gamma_{n+2}(\Delta)$ for some integer $s \geq 1$. It then follows via Proposition~\ref{nilq} that $\nu$ centralises $V_0$. Hence, it will suffice to show that $\nu$ centralises $V_1$, and we may assume without loss of generality that $V=V_1$ and the minimal polynomial $m(x)$ of $\zeta$ acting on $V$ is coprime to $(x-1)$.

As before, we may write $\operatorname{ord}(\rho_V(\zeta)) = p^l s$ for some integers $l,s \in \mathbb{N}$ with $s$ coprime to $p$. Again, $\gamma_m(\Delta)$ is fixed set-wise by the conjugation action of $\Delta^\mathrm{ab}$, so it forms a $\mathbb{Z}[x]$-submodule of $D$ and we may consider the $\mathbb{F}_p[\Delta^\mathrm{ab}]$-module
\[
W = \frac{D}{pD+\gamma_{p^l+2}(\Delta)}
\] which forms an $\mathbb{F}_p$-vector space of dimension $p^ln$. Note that the minimal polynomial of $\zeta$ on $W$ is $(x-1)^{p^l}$, as per Remark~\ref{Rem::ModuleDescription}. It follows that the minimal polynomial of $\zeta$ on $W$ is coprime to $m(x)$, and the primary decomposition theorem yields an isomorphism
\[
S:=\frac{D}{\left((pD+\gamma_{p^l+2}(\Delta)\right) \cap L} \cong W \oplus V
\] as $\Z[x]$-modules. Writing $ A = \operatorname{GL}(W)$ and $B = \operatorname{GL}(V)$ as well as $a = \rho_W(\zeta^r)$ and $b = \rho_V(\zeta^r)$ and $b' = \rho_V(\nu)$, we find that $\operatorname{ord}(a) = \operatorname{ord}(b) = p^l$ while $\operatorname{ord}(b') = p^k$. Moreover, the element $\nu \in N$ acts trivially on $W$ by Proposition~\ref{nilq}. It follows via Lemma~\ref{nocyclic} that the images of $\zeta^r$ and $\nu$ in $\operatorname{Aut}(S) \cong A \times B$ generate a subgroup of $\operatorname{Aut}(T)$ isomorphic to $C_{p^l} \times C_{p^k}$, a contradiction to Lemma~\ref{cyclic}. We conclude that any $\nu \in N$ must centralise $V$ and the proof is complete.
\end{proof}

\begin{lemma}\label{Lem::Abelian} The group $N$ is abelian.
\end{lemma}

\begin{proof} We commence by exhibiting a cofinal family of finite quotients of the torsion-free lamplighter group $\Gamma = \Z^n \wr \Z$. Consider the semidirect product decomposition of $\Gamma$ in (\ref{Eq::DefOfLL}) and fix generators $a_1, \ldots, a_n$ for the free $\Z[\langle t \rangle]$-module $M$, so that $\Gamma$ is as a group by $a_1, \ldots, a_n$ and $t$. Given a positive integer $m \in \mathbb{N}$, define
\[
\Gamma(m) = C_m^n \wr C_m \cong R_m^n \rtimes C_m
\] where $R_m$ is the additive abelian group of the ring
\[
R_m = \frac{(\Z/m\Z)[x]}{(x^m - 1)}
\]
and the generator $\tau_m$ of $C_m$ acts on each copy of $R_m$ as multiplication by $x$. The natural maps $M \twoheadrightarrow R_m$ and $\langle t \rangle \twoheadrightarrow C_m$ yield an epimorphism $\pi_m \colon \Gamma \twoheadrightarrow C_m$. In fact, given any finite-index normal subgroup $U \trianglelefteq \Gamma$, setting $m = 2\cdot [\Gamma : U]$ and $V = \langle \langle t^m,a_1^m,\ldots, a_n^m\rangle \rangle \trianglelefteq \Gamma$ yields $V \leq U$ and $\Gamma / V \cong \Gamma(m)$. Hence the set $\{\Gamma(m) : m \in \mathbb{N}\}$ is a cofinal subsystem of the inverse system of finite quotients $\mathcal{C}(\Gamma) = \mathcal{C}(\Delta)$. To prove that $N$ is abelian, it thus suffices to show that $\pi_m(N) \leq \Gamma(m)$ is abelian for $m > 2$.
Moreover, Lemma~\ref{centralizer} yields $\pi_m(N) \leq C_{\Gamma(m)}([\Gamma(m), \Gamma(m)])$, so in fact it will suffice to show that the centraliser $Z(m) = C_{\Gamma(m)}([\Gamma(m), \Gamma(m)]) \trianglelefteq \Gamma(m)$ is abelian for each positive integer $m > 2$.

Indeed, we shall demonstrate that $Z(m) = R_m^n$, which is abelian by construction. On the one hand, we have $[\Gamma(m), \Gamma(m)] = (x-1) R_m^n$ which is centralised by its abelian supergroup $R_m^n$, so $R_m^n \leq Z(m)$. Conversely, suppose for a contradiction that there exists $z \in Z(m) \backslash R_m^n$. We may assume without loss of generality that $z=\tau^k$ for some $k \in \{1, \ldots, m-1\}$. The condition that $z$ centralizes $[\Gamma(m),\Gamma(m)]$ is equivalent to $\tau^k-1$ annihilating $(x-1)R_m^n$, which in turn is equivalent to the equation $(x^k-1)(x-1)=0$ in $R_m$. Now, if $k \neq 1$ then
\[
(x^k-1)(x-1)=x^{k+1}-x^k-x+1 \neq 0
\] because $x \not \in \{1,x^k,x^{k+1}\}$ and $1,x, \ldots x^{m-1}$ are a basis of $R_m$ as a free $\Z/m\Z$-module. On the other hand, if $k = 1$ then \[
(x^k-1)(x-1)=x^2 -2x+1 \not =0
\]
because $m>2$. This contradicts $z \in Z(m)$, so we conclude that $Z(m) = R_m^n$ is abelian and the proof is complete.
\end{proof}

The conjunction of Lemma~\ref{Lem::Abelian} and Lemma~\ref{centralizer} yields Proposition~\ref{Prop::LiftingCentralisers}.

\section{The Proof of Theorem~\ref{Thm::MainThm}} \label{Sec::Final}

In this section, we combine Proposition~\ref{Prop::LiftingCentralisers} and Theorem~\ref{Thm::NoetherianDomains} to prove Theorem~\ref{Thm::MainThm}.

\MainThm*

Indeed, $\Delta$ be a finitely generated residually finite group with admitting an isomorphism of profinite completions $\phi \colon \hD \to \hG$. By Proposition~\ref{Prop::LiftingCentralisers}, there exists a normal abelian subgroup $N \trianglelefteq \Delta$ such that $N \leq C_\Delta([\Delta,\Delta])$ and $\Delta / N \cong \Z$ is an infinite cyclic group generated by the image of some $\zeta \in \Delta$. The abelian subgroup $N$ then acquires the structure of a module over the ring $\Lambda = \Z[x^\pm]$ where $x$ acts as conjugation by $\zeta$ within the supergroup $\Delta$. Likewise, the decomposition (\ref{Eq::DefOfLL}) gives $\Gamma \cong M \rtimes \Z$ where $M$ is a free $\Lambda$-module of rank $n$ and $x$ acts as conjugation by $t$ within $\Gamma$. We commence with the following lemma.
\begin{lemma}\label{Lem::ImageOfAbelianSubgroup}
    Let $\overline{N}$ and $\overline{M}$ denote the closures of $N$ and $M$ within the profinite completions $\hD$ and $\hG$, respectively. The isomorphism of profinite groups $\phi \colon \hD \to \hG$ must satisfy $\phi(\overline{N}) = \overline{M}$.
\end{lemma}
\begin{proof}
    To begin, we shall demonstrate that the inclusion $\phi(\overline{N}) \leq \overline{M}$ holds. Assume for a contradiction that this is not the case, so $\overline{M} \lneqq \phi(\overline{N}) \cdot \overline{M}$ holds. By \cite[Proposition 3.2.5]{Ribes2010}, there is an isomorphism of profinite groups
    \[
    G := \hG / \overline{M} \cong \widehat{\Gamma / M} \cong \pZ 
    \] and $G$ has a nontrivial subgroup $H := \phi(\overline{N}) \cdot \overline{M} / \overline{M}$. By \cite[Theorem 2.7.2]{Ribes2010}, there exists a prime number $p$ and a natural number $s \in \mathbb{N}$ such that the $p$-Sylow subgroup $H_p \trianglelefteq H$ is non-trivial and forms a subgroup of the $p$-Sylow subgroup $G_p \trianglelefteq G$ with index $[G_p : H_p] = p^s$. Let $\widetilde{H_p}$ be the preimage of $H_p$ in $\hG$. Now Proposition~\ref{Prop::LiftingCentralisers} tells us that the subgroup $N$ centralises $[\Delta,\Delta]$ in $\Delta$, so the image $\phi(\overline{N})$ must centralise the commutator $\phi(\overline{[\Delta,\Delta]}) = [\hG,\hG]$ in the profinite group $\hG$. Similarly, $M \leq C_\Gamma([\Gamma, \Gamma])$ and thus $\overline M \leq C_{\hG}([\hG, \hG])$ holds. It follows that $\phi(\overline{N}) \cdot \overline{M}$ centralises the commutator $[\hG, \hG]$, a property which passes to its subgroup $\widetilde{H_p}$. Thus, for any finite quotient $\Gamma \twoheadrightarrow Q$, the composition
    \[
    \varpi \colon \Gamma \twoheadrightarrow Q \twoheadrightarrow Q / C_Q([Q,Q]) =: T
    \]
    yields an epimorphism of $p$-Sylow subgroups $\varpi(G_p / H_p) \twoheadrightarrow T_p$. In particular, there exists a prime number $p$ such that $|T_p| \leq [G_p : H_p] = p^s$, i.e. the order of $p$-Sylow subgroups is bounded among the set of groups $Q / C_Q([Q,Q])$ where $Q$ ranges through the finite quotients of $\Gamma$. This contradicts the existence of the quotient $\Gamma \twoheadrightarrow \mathbb{F}_p[C_{p^l}] \rtimes C_{p^l} =: Q_{p,l}$ for any prime $p$ and any $l \in \mathbb{N}$.
    
    We infer that the inclusion $\phi(\overline{N}) \leq \overline{M}$ holds. Thus we obtain a commutative diagram with exact rows
    \[
    \begin{tikzcd}
1 \arrow[r] & \overline{N} \arrow[d] \arrow[r] & \widehat{\Delta} \arrow[d, "\phi"] \arrow[r] & \widehat{\Delta / N} \arrow[d, "\widetilde{\phi}"] \arrow[r] & 1 \\
1 \arrow[r] & \overline{M} \arrow[r]           & \widehat{\Gamma} \arrow[r]                   & \widehat{\Gamma / M} \arrow[r]           & 1
\end{tikzcd}
    \]
    via \cite[Proposition 3.2.5]{Ribes2010}. The induced map $\widetilde{\phi}$ is a surjection by the five lemma. But Proposition~\ref{Prop::LiftingCentralisers} yields $\widehat{\Delta / N} \cong \widehat\Z \cong \widehat{\Gamma / M}$ and finitely generated profinite groups are Hopfian \cite[Lemma 2.5.2]{Ribes2010}, so $\widetilde{\phi}$ must in fact be an isomorphism. Another application of the five lemma yields $\phi(\overline{N}) = \overline{M}$, as postulated.
\end{proof}

Moreover, the closures $\overline{M}$ and $\overline{N}$ of $M$ and $N$ within the respective profinite completions $\hG$ and $\hD$ acquire the structure of profinite $\widehat{\Lambda}$-modules via conjugacy within the ambient profinite groups. By \cite[Lemma 5.2]{Wykowski2025_2}, these $\widehat{\Lambda}$-module structures on $\overline{N}$ and $\overline{M}$ are isomorphic to the $\Lambda$-profinite completions of the $\Lambda$-modules $M$ and $N$, which we shall denote as $\hM$ and $\hN$, respectively. It then follows via Lemma~\ref{Lem::ImageOfAbelianSubgroup} that there is a commutative diagram
\[
\begin{tikzcd}
1 \arrow[r] & \overline{N} \cong \hN \arrow[d, "f"] \arrow[r] & \widehat{\Delta} \arrow[d, "\phi"] \arrow[r] & \widehat{\Delta / N} \cong \pZ \arrow[d, "\widetilde{\phi}"] \arrow[r] & 1 \\
1 \arrow[r] & \overline{M} \cong \hM \arrow[r]           & \widehat{\Gamma} \arrow[r]                   & \widehat{\Gamma / M} \cong \pZ \arrow[r]           & 1
\end{tikzcd}
\]
where $\widetilde{\phi}$ is the induced automorphism of the profinite group $\pZ$, and $f$ is the isomorphism of profinite abelian groups obtained by restricting $f = \phi \at{\overline{N}} : \overline{N} \to \overline{M}$. It follows from the commutativity of this diagram that $f$ forms a twisted isomorphism of profinite $\widehat{\Lambda}$-modules, in the sense that
\begin{equation}
    f(\lambda \cdot a) = \widetilde{\phi}(\lambda) \cdot f(a)
\end{equation}
holds for all $a \in \hN$ and $\lambda \in \widehat{\Lambda}$. Consider now the composition of isomorphisms of profinite abelian groups given by 
\[
F \colon \widehat{\Lambda}^n \xrightarrow{\,(\widetilde{\phi}^{-1})^n\,} \widehat{\Lambda}^n \xrightarrow{\quad \sim\quad} \hM \xrightarrow{\quad f \quad} \hN
\]
where the central map arises as the completion of a choice of isomorphism witnessing that the $\Lambda$-module $M$ is free of rank $n$. Observe that for any $\lambda \in \widehat{\Lambda}$ and $\vec{\omega} \in \widehat{\Lambda}^n$, we have
\begin{align*}
    F(\lambda \cdot \vec{\omega}) &= f \left((\widetilde\phi^{-1})^n(\lambda \cdot \vec{\omega}) \right) \\
    &= f \left(\widetilde{\phi}^{-1}(\lambda)\cdot (\widetilde{\phi}^{-1})^n(\vec{\omega}) \right) \\
    &= \widetilde{\phi} \left(\widetilde{\phi}^{-1}(\lambda)\right) \cdot f \left((\widetilde{\phi}^{-1})^n(\vec{\omega}) \right) \\
    &= \lambda \cdot F(\vec{\omega})
\end{align*}
so in fact $F$ forms an isomorphism of profinite $\widehat\Lambda$-modules. It follows that the $\Lambda$-module $N$ is profinitely isomorphic to the free module of rank $n$. But free $\Lambda$-modules are $\Lambda$-profinitely rigid by the conjunction of Theorem~\ref{Thm::NoetherianDomains} and Theorem~\ref{Thm::Quillen-Suslin}. Thus we conclude that $M \cong \Lambda^n$ and there is an isomorphism
\[
\Delta \cong N \rtimes_x \Z \cong \Lambda^n \rtimes_x \Z \cong M \rtimes_x \Z \cong \Gamma
\]
as discrete groups. We have demonstrated that $\Gamma$ is profinitely rigid in the absolute sense, and the proof of Theorem~\ref{Thm::MainThm} is complete.

\printbibliography

\end{document}